\documentclass[a4paper,11pt]{amsart}
\usepackage{amsmath}
\usepackage{amsthm}
\usepackage{amsfonts}
\usepackage{amssymb}
\usepackage{latexsym}
\usepackage{mathrsfs}

\usepackage{enumitem}

\usepackage[section]{placeins}
\usepackage{xcolor}

\usepackage[abs]{overpic}
\usepackage[titletoc]{appendix}
\usepackage{graphicx}
\usepackage{latexsym}
\usepackage[utf8]{inputenc}
\usepackage{epsfig}
\usepackage{psfrag}

\usepackage{tikz}
\usepackage[normalem]{ulem}

\usepackage{diagbox}

\numberwithin{equation}{section}

\newtheorem{theorem}{Theorem}[section]{\bf}{\it}
\newtheorem{lemma}[theorem]{Lemma}{\bf}{\it}
\newtheorem{proposition}[theorem]{Proposition}{\bf}{\it}
\newtheorem{corollary}[theorem]{Corollary}{\bf}{\it}
{\bf}{\it} 
{\bf}{\it}
\newtheorem*{theorem*}{Theorem}

\newtheorem{remark}[theorem]{Remark}
{\bf}{\it}
{\bf}{\it}

\newtheorem*{namedtheorem}{\theoremname}
\newcommand{\theoremname}{testing}
\newenvironment{named}[1]{\renewcommand{\theoremname}{#1}\begin{namedtheorem}}{\end{namedtheorem}}

\newtheorem*{definition*}{Definition}

{\bf}{\it}
{\bf}{\it}
{\bf}{\it}
\newtheorem{example}[theorem]{Example}
\newtheorem*{example*}{Example}

\theoremstyle{remark}

\theoremstyle{definition}

\theoremstyle{remark}

\newcommand{\R}{\mathbb R}
\newcommand{\Z}{\mathbb Z}
\newcommand{\C}{\mathbb C}
\newcommand{\N}{\mathbb N}

\newcommand{\loc}{{\operatorname{loc}}}

\newdimen\vintkern\vintkern11pt
\def\vint{-\kern-\vintkern\int}

\newcommand{\dn}{\;\mathrm{d}}
\newcommand{\norm}[1]{\lVert #1 \rVert}

\newcommand{\bR}{\mathbb{R}}
\newcommand{\bS}{\mathbb{S}}

\newcommand{\cF}{\mathcal{F}}

\newcommand{\vol}{\mathrm{vol}}

\newcommand{\dR}{\mathrm{dR}}
\newcommand{\bT}{\mathbb{T}}

\newcommand{\interior}{\mathrm{int}}

\newcommand{\ccL}{\mathscr{L}}
\newcommand{\ccR}{\mathscr{R}}
\newcommand{\cW}{\mathscr W}

\newcommand{\cH}{\mathcal{H}}

\newcommand{\tor}{\mathrm{Tor} \,}

\newcommand{\abs}[1]{\left\lvert #1 \right\rvert}

\renewcommand{\le}{\leqslant}
\renewcommand{\ge}{\geqslant}

\title[Quasiregularly elliptic manifolds]{De Rham algebras of closed quasiregularly elliptic manifolds are Euclidean}

\author{Susanna Heikkilä}
\address{Department of Mathematics and Statistics, P.O. Box 68 (Pietari Kalmin katu 5), FI-00014 University of Helsinki, Finland}
\email{susanna.a.heikkila@helsinki.fi}

\author{Pekka Pankka}
\address{Department of Mathematics and Statistics, P.O. Box 68 (Pietari Kalmin katu 5), FI-00014 University of Helsinki, Finland}
\email{pekka.pankka@helsinki.fi}

\thanks{This work was supported in part by the Academy of Finland project \#332671. This work is also supported in part by the National Science Foundation under Grant No. DMS-1928930 while P.P. participated in a program hosted by the Mathematical Sciences Research Institute (MSRI) in Berkeley, California, during the
Spring 2022 semester.
}

\begin{document}

\begin{abstract}
We show that, if a closed, connected, and oriented Riemannian $n$-manifold $N$ admits a non-constant quasiregular mapping from the Euclidean $n$-space $\mathbb R^n$, then the de Rham cohomology algebra $H_{\dR}^*(N)$ of $N$ embeds into the exterior algebra ${\bigwedge}^*\mathbb R^n$. As a consequence, we obtain a homeomorphic classification of closed simply connected quasiregularly elliptic $4$-manifolds.
\end{abstract}

\maketitle

\section{Introduction}

In this article we prove the following result on the de Rham cohomology of quasiregularly elliptic manifolds.

\begin{theorem}
\label{thm:main}
Let $N$ be a closed quasiregularly elliptic $n$-manifold, $n\ge 2$. Then there exists an embedding of graded algebras $H^*_{\dR}(N) \to \bigwedge^* \R^n$.
\end{theorem}

A connected and oriented Riemannian $n$-manifold $N$ is called \emph{quasiregularly elliptic} if there exists a non-constant quasiregular mapping $\R^n \to N$. A continuous map $f\colon M\to N$ between oriented Riemannian $n$-manifolds, $n\ge 2$, is \emph{quasiregular} if $f$ belongs to the Sobolev space $W_{\loc}^{1,n}(M,N)$ and there exists a constant $K\ge 1$ for which 
\begin{equation}
\label{eq:QR}
\norm{Df}^n \le K J_f \quad \text{ a.e. in } M;
\end{equation}
here $\norm{Df}$ is the operator norm of the weak differential $Df$ and $J_f$ is the Jacobian determinant of $f$. 

Theorem \ref{thm:main} may be viewed as a quasiconformal analog of a classical result in degree theory: \emph{For a non-zero degree mapping $f\colon M\to N$ between orientable closed manifolds, the induced map $f^* \colon H^*(N) \to H^*(M)$ in cohomology is an embedding of graded algebras.} This analogy can be taken further by observing that, if there exists a piecewise linear branched covering $\bT^n \to N$, which is quasiregular with respect to smooth structures of $\bT^n$ and $N$, then $N$ is a quasiregularly elliptic.

Theorem \ref{thm:main} has an antecedent in Kangasniemi's theorem \cite{Kangasniemi-Adv}: \emph{If a closed, connected, and oriented Riemannian $n$-manifold $N$ admits a non-constant and non-injective uniformly quasiregular mapping $N \to N$, then there exists an embedding of graded algebras $H^*_{\dR}(N) \to \bigwedge^* \R^n$}. Regarding Kangasniemi's theorem, we recall that a quasiregular self-map is \emph{uniformly quasiregular} if the distortion inequality \eqref{eq:QR} holds for all the iterates of the map with the same constant. We also note that such closed manifolds are quasiregularly elliptic; see Martin, Mayer, and Peltonen \cite{Martin-Mayer-Peltonen}. 

In terms of known results on the cohomology of closed quasiregularly elliptic manifolds, Theorem \ref{thm:main} extends a theorem of Prywes \cite{PR}: \emph{Let $N$ be a closed quasiregularly elliptic $n$-manifold. Then $\dim H^k_\dR(N) \le \dim \bigwedge^k \R^n = \binom{n}{k}$}. This dimension bound for the de Rham cohomology theorem is sharp as the $n$-torus $\bT^n$ is quasiregularly elliptic. Prywes's theorem answers a question of Gromov \cite[p.~200]{GromovM:Hypmga} whether there exists closed simply connected manifolds, which are not quasiregularly elliptic. This is indeed the case as, for example, $\#^4(\bS^2\times \bS^2)$ is not quasiregularly elliptic by Prywes's theorem. We refer to Bonk and Heinonen \cite{BonkM:Quamc} for a previous result on distortion dependent bounds on the de Rham cohomology of quasiregularly elliptic manifolds. We also recall a theorem of Peltonen \cite{Peltonen} which shows that connected sums $\bT^n \# N$, where $N$ is not a rational homology sphere, are not quasiregularly elliptic. 

We also note in passing that dimension bounds on the de Rham cohomology in Prywes's theorem do not hold for open quasiregularly elliptic manifolds. Indeed, for any finite set $P\subset \bS^n$ there exists a quasiregular mapping $f\colon \R^n \to \bS^n$ for which $f\R^n = \bS^n \setminus P$; see \cite{DP}. Thus, it is known that the $(n-1)$-cohomology of an open quasiregularly elliptic manifold $N$ may have arbitrarily large finite dimension. The dependence of the dimension of $H_{\dR}^k(N)$ on the distortion constant is not known in general.

\subsection{Classification of closed simply connected quasiregularly elliptic $4$-manifolds}

Our main application of Theorem \ref{thm:main} is a classification of closed simply connected quasiregularly elliptic $4$-manifolds.

In dimensions $n=4m$, Theorem \ref{thm:main} imposes upper bounds for the middle-dimensional Betti numbers $\beta^+_{2m}(N)$ and $\beta^-_{2m}(N)$ given by positive and negative definite subspaces of the intersection form $H^{2m}(N;\R) \times H^{2m}(N;\R) \to \R$ of a closed quasiregularly elliptic manifold $N$. In particular, in dimension $n=4$, both Betti numbers $\beta^+_2(N)$ and $\beta^-_2(N)$ are at most three. Therefore the classification of smooth, closed, simply connected, and oriented $4$-manifolds yields that $N$ is homeomorphic to one of the manifolds
\begin{equation} \label{eq:list}
\#^k (\bS^2\times \bS^2) \quad \text{or} \quad \#^j \C P^2 \#^i \overline{\C P^2},
\end{equation}
where $k,j,i\in \{0,1,2,3\}$ and $\#^0 (\bS^2\times \bS^2)=\#^0 \C P^2 \#^0 \overline{\C P^2}=\bS^4$.

Piergallini and Zuddas show in \cite{PZ} that these manifolds are quasiregularly elliptic. Indeed, by a result Piergallini and Zuddas \cite{PZ}, for the manifolds $N$ in \eqref{eq:list} there exists a piecewise linear branched covering $\bT^4 \to N$, where $\bT^4$ and $N$ have the standard PL structure compatible with the usual smooth structure. Thus composing a such branched covering $\bT^4 \to N$ with a locally isometric Riemannian covering map $\R^4 \to \bT^4$, we obtain the required quasiregular mapping $\R^4 \to N$. 

Combining Theorem \ref{thm:main} with this result of Piergallini and Zuddas we obtain the following classification.

\begin{corollary}
\label{cor:classification}
A closed simply connected $4$-manifold is quasiregularly elliptic if and only if it is homeomorphic to either $\#^k(\bS^2\times \bS^2)$ or $\#^j \C P^2 \#^i \overline{\C P^2}$ for some $k,j,i\in \{0,1,2,3\}$. 
\end{corollary}

Similar classifications of quasiregularly elliptic manifolds, based on general classification theorems for manifolds, are known in dimensions two and three. In dimension two, the result is classically deduced from uniformization of Riemann surfaces and the Sto\"ilow theorem for quasiregular mappings. In dimension three, the classification follows from Perelman's solution to the Geometrization conjecture and a growth bound for the fundamental group. This reduction of the classification to the Geometrization conjecture is due to Jormakka \cite{JormakkaJ:Quamf3}. We recall briefly these results from the point of view of Prywes's theorem and Theorem \ref{thm:main}.

In dimension $n=2$, a closed quasiregularly elliptic $2$-manifold $N$ satisfies $\dim H^1_{\dR}(N)\le 2$. Since $N$ is oriented, we immediately observe that $N$ is homeomorphic to either $\bS^2$ or $\bT^2$ by classification of orientable closed surfaces.
 
For a closed quasiregularly elliptic $3$-manifold $N$, we deduce from Prywes's theorem that $\dim H^1_\dR(N) = \dim H^2_\dR(N) \le 3$. By elementary algebra, we obtain from Theorem \ref{thm:main} that $\dim H^1_\dR(N) \ne 2$. Thus $N$ has the de Rham cohomology of one of the manifolds $\bS^3$, $\bS^2\times \bS^1$, or $\bT^3$. This homological result is sharp in the sense that, by Jormakka's classification theorem \cite{JormakkaJ:Quamf3}, the only closed quasiregularly elliptic $3$-manifolds are $\bS^3$, $\bS^2\times \bS^1$, $\bT^3$, and their quotients.
We find it interesting that, in fact, all closed quasiregularly elliptic manifold in dimensions two and three admit a branched covering from a torus.

In dimension four, to our knowledge, the only known examples of closed simply connected quasiregularly elliptic $4$-manifolds, before the result of Piergallini and Zuddas, were $\bS^4$, $\bS^2\times \bS^2$, $\#^2(\bS^2\times \bS^2)$, and $\C P^2$; here the first two are given by elementary maps and the latter two are given by constructions of Rickman \cite{RI2} and Luisto--Prywes \cite{Luisto-Prywes}, respectively.

\subsection{Quasiregular mappings from a ball}

Theorem \ref{thm:main} admits a reformulation for quasiregular mappings from a unit ball into closed manifolds. This reformulation in turn yields a cohomological Montel theorem for quasiregular mappings satisfying a doubling condition. For the statements, let $N$ be an oriented Riemannian $n$-manifold, let $K\ge 1$ and $D>1$ be constants, and let $\cF_{K,D}(N)$ be a family of all non-constant quasiregular mappings $f\colon B_2^n \to N$ satisfying the distortion inequality \eqref{eq:QR} with constant $K$ and an additional doubling condition
\[
\int_{B_2^n} f^* \vol_N \le D\int_{B^n} f^* \vol_N.
\]
Here and in what follows, $B^n$ and $B_2^n$ denote the Euclidean balls of radius $1$ and $2$ centered at the origin, respectively. In what follows, we also denote
\[
A(f) = \int_{B^n} f^* \vol_N.
\]

\begin{theorem}
\label{thm:sup}
Let $N$ be a closed, connected, and oriented Riemannian $n$-manifold, $n\ge 2$, for which
\[
\sup_{f\in \cF_{K,D}(N)} A(f) = \infty.
\]
Then there exists an embedding of graded algebras $H^*_{\dR}(N) \to \bigwedge^* \R^n$.
\end{theorem}

We conclude that, if $H^*_\dR(N)$ is not a subalgebra of $\bigwedge^* \R^n$, there exists a uniform upper bound for the area $A(f)$, and hence also for the energy, of mappings $f$ in $\cF_{K,D}(N)$. This uniform upper bound replaces the degree bound in the Gromov's compactness theorem for quasiregular maps \cite{Pankka-Souto} and, as a consequence of these two results, we obtain the following cohomological version of Montel's theorem for quasiregular maps in $\cF_{K,D}(N)$; see Rickman \cite[Corollary IV.3.14]{RI1} for Rickman's Montel theorem for quasiregular mappings without an additional doubling assumption.

\begin{corollary}
\label{cor:sup}
Let $N$ be a closed, connected, and oriented Riemannian $n$-manifold, $n\ge 2$. If $H^*_{\dR}(N)$ is not a subalgebra of $\bigwedge^* \R^n$, then $\overline{\cF_{K,D}(N)}$ is compact with respect to the topology of local uniform convergence. In particular, in this case, $\cF_ {K,D}(N)$ is a normal family.
\end{corollary}

We finish this introduction by stating a more general result on weak limits of normalized Jacobians of mappings in $\cF_{K,D}(N)$. This statement may be viewed as non-dynamical version of Kangasniemi's absolute continuity theorem for invariant measures for uniformly quasiregular mappings \cite{Kangasniemi-AJM}. For the statement, for each $f\in \cF_{K,D}(N)$, let $\nu_f$ be the measure
\[
\nu_f = \frac{1}{A(f)} f^*\vol_N.
\]
We call $\nu_f$ the \emph{normalized measure of $f$}. Note that, clearly, $\nu_f \ll \vol_{\R^n}$.

\newcommand{\weakto}{\rightharpoonup}

For a closed manifold $N$, which is not a rational homology sphere, we have the following dichotomy for sequences in $\cF_{K,D}(N)$; for the statement, note that a sequence in $\cF_{K,D}(N)$ need not have a locally uniformly converging subsequence.

\begin{theorem}
\label{thm:measures}
Let $N$ be a closed, connected, and oriented Riemannian $n$-manifold, $n \ge 2$, which is not a rational homology sphere. 
Then a sequence $(f_j)$ in $\cF_{K,D}(N)$ has a subsequence $(f_{j_i})$ for which either $(f_{j_i})$ converges to a constant map or the normalized measures $\nu_{f_{j_i}}$ converge vaguely to a measure $\nu \ll \vol_{\R^n}$ in $B_2^n$.
\end{theorem}

Theorem \ref{thm:measures} is sharp in the sense that, for $N=\bS^n$, there exists a sequence of mappings $(f_j \colon B_2^n \to \bS^n)$ for which the sequence $(\nu_{f_j})$ of normalized measures converges vaguely to a measure $\nu$, which is singular with respect to the Lebesgue measure. Recall that measures $(\nu_{f_j})$ \emph{converge vaguely to $\nu$} if $\int \varphi \nu_{f_j} \to \int \varphi \nu$ for each $\varphi \in C_0(B_2^n)$.

\begin{remark}
Theorem \ref{thm:measures} holds also for covering spaces of closed manifolds which are not rational homology spheres. In particular, Theorem \ref{thm:measures} holds for quasiregular mappings in $\cF_{K,D}(\R^n)$.
\end{remark}

\subsection{Idea of the proof of Theorem \ref{thm:sup}}

Since Theorem \ref{thm:main} reduces to Theorem \ref{thm:sup}, we discuss now the ideas in terms of Theorem \ref{thm:sup}. To that end, let $N$ be a closed, connected, and oriented Riemannian $n$-manifold and let 
\[
h \colon H^*_\dR(N) \to \Omega^*(N), \quad c\mapsto h_c,
\]
be the mapping which associates to each de Rham class $c\in H^*_\dR(N)$ its unique harmonic representative $h_c \in c$. The mapping $h$ is typically only a graded linear mapping and not an algebra homomorphism, since harmonic forms typically do not form an algebra. We refer to Kotschick \cite{Kotschick} for a discussion on so-called \emph{formal Riemannian manifolds}, whose harmonic forms form an algebra, and to Kangasniemi \cite{Kangasniemi-Adv} for the role of formal measurable Riemannian metrics in the uniformly quasiregular dynamics.

We define, for each $f\in \cF_{K,D}(N)$, a normalized pull-back 
\[
f^\# \colon H^*_\dR(N) \to \cW^*(B_2^n), \quad H^k_{\dR}(N)\ni c\mapsto A(f)^{-\frac{k}{n}}f^*(h_c),
\]
where $A(f)^{-\frac{k}{n}}=1$ for $k=0$ and $\cW^*(B_2^n)$ is the graded algebra
\[
\cW^k(B_2^n) = \left\{ \begin{array}{ll}
W^{d,\frac{n}{k}}(B_2^n; \bigwedge^k \R^n), & \text{for } 0<k\le n \\
\ker (d \colon \Omega^0(B_2^n) \to \Omega^1(B_2^n)), & \text{for } k =0
\end{array}\right.
\]
of Sobolev forms on $B_2^n$. In particular, $\cW^0(B_2^n)$ consists merely of constant functions on $B_2^n$. Here $W^{d,\frac{n}{k}}(B_2^n;\bigwedge^k \R^n)$ is the partial Sobolev space of $\frac{n}{k}$-integrable $k$-forms on $B_2^n$; see Section \ref{sec:prelims} for details.

It is easy to observe that $f^\#$ is an injective graded linear mapping and not an algebra homomorphism unless $h$ is an algebra homomorphism.  Nevertheless, if $(f_j)$ is such a sequence in $\cF_{K,D}(N)$ that $A(f_j) \to \infty$ and that the sequence $(f_j^\#|_{H_{\dR}^k(N)} \colon H_{\dR}^k(N)\to \cW^k(B_2^n))$ converges weakly for each $k=1,\ldots,n-1$, then the limit operators yield an embedding of graded algebras.

\begin{theorem}
\label{thm:limit}
Let $N$ be a closed, connected, and oriented Riemannian $n$-manifold, $n\ge 2$. Let $(f_j)$ be a sequence in $\cF_{K,D}(N)$ for which $A(f_j) \to \infty$ and $f_j^\# \weakto L$, where
\[
f_j^\#, L \colon \bigoplus_{k=1}^{n-1} H^k_{\dR}(N) \to \bigoplus_{k=1}^{n-1} L^\frac{n}{k}(B_2^n; {\bigwedge}^k \R^n).
\]
Then $L$ extends to an embedding of graded algebras 
\[L\colon H^*_\dR(N) \to \cW^*(B_2^n)
\]
satisfying
\[
\int_{B^n} L([\vol_N])=1.
\]
\end{theorem}

Having Theorem \ref{thm:limit} at our disposal, Theorem \ref{thm:sup} follows by showing that, for each sequence $(f_j)$ satisfying $A(f_j) \to \infty$, there exists a subsequence whose normalized pull-backs converge weakly for each $H_{\dR}^k(N)$, $k=1,\ldots,n-1$, to a linear map $L\colon H_{\dR}^k(N)\to  L^\frac{n}{k}(B_2^n; {\bigwedge}^k \R^n)$. Since these linear maps yield an embedding of algebras $L \colon H_{\dR}^*(N) \to \cW^*(B_2^n)$, it suffices then to observe that there exists a set of positive measure $E\subset B_2^n$ for which the map
\[
H^*_\dR(N) \to {\bigwedge}^* \R^n, \quad c\mapsto (L(c))(x_0),
\]
is a well-defined embedding of graded algebras for $x_0 \in E$.

\subsection{Organization of the article}

This article is organized as follows. 

In Section \ref{sec:reduction}, we give the reduction of Theorem \ref{thm:main} to Theorem \ref{thm:sup}.  After this, in Section \ref{sec:prelims}, we discuss the normalized pull-backs and the Sobolev algebra $\cW^*$. In Sections \ref{sec:potential} and \ref{sec:extension}, we prove a preliminary version of Theorem \ref{thm:limit} which shows that the limit extends to an algebra homomorphism. In Section \ref{sec:monomorphism}, we prove the embedding part of Theorem \ref{thm:limit}. In Section \ref{sec:sup}, we finalize the proof of our main theorem, Theorem \ref{thm:sup}.

In Section \ref{sec:example}, we discuss a sequence of quasiregular mappings $B_2^n \to \bS^n$ showing the sharpness of Theorem \ref{thm:measures}. In Section \ref{sec:measures}, we prove Theorem \ref{thm:measures}. In Appendix \ref{app:classification}, we briefly recall the classification of smooth, closed, simply connected, and oriented $4$-manifolds. Finally, in Appendix \ref{app:PL}, we discuss briefly the quasiregularity of the piecewise linear branched covers $\bT^4 \to N$ of Piergallini and Zuddas.

\subsection{Acknowledgements}
We thank Ilmari Kangasniemi, Jani Onninen, Eden Prywes, and Juan Souto for discussions on these topics over the years. We also thank Eero Hakavuori and Toni Ikonen for a helpful remark at the right time.

\section{Reduction of Theorem \ref{thm:main} to Theorem \ref{thm:sup}}
\label{sec:reduction}

In this section, we discuss the reduction of Theorem \ref{thm:main} to Theorem \ref{thm:sup}. Since both Theorems \ref{thm:main} and \ref{thm:sup} hold trivially if the target manifold is a rational homology sphere, the reduction is based on the following result; see also \cite[Section 4]{PR}.

For the statement, we say that a quasiregularly elliptic manifold $N$ is \emph{$K$-quasiregularly elliptic} if there exists a non-constant quasiregular mapping $\R^n \to N$ satisfying the distortion inequality \eqref{eq:QR} with constant $K$.

\begin{proposition}
\label{prop:reduction}
Let $N$ be a closed $K$-quasiregularly elliptic $n$-manifold, $n\ge 2$, which is not a rational homology sphere. Then there exists a constant $D=D(n)>1$ for which
\[
\sup_{f\in \cF_{K,D}(N)} A(f)=\infty.
\]
\end{proposition}

We deduce Proposition \ref{prop:reduction} from the following two results. The first result is a lower growth bound for quasiregular mappings into manifolds with non-trivial homology due to Bonk and Heinonen \cite{BonkM:Quamc}; see also \cite{Heikkila} for the corresponding result in the setting of quasiregular curves.

\begin{theorem}{\cite[Theorem 1.11]{BonkM:Quamc}}
\label{thm:fast-growth}
Let $N$ be a closed, connected, and oriented Riemannian $n$-manifold, $n\ge 2$, and let $f\colon \R^n \to N$ be a non-constant quasiregular map satisfying the distortion inequality \eqref{eq:QR} with constant $K$. If $N$ is not a rational homology sphere, then there exists a constant $\varepsilon=\varepsilon(n,K)>0$ satisfying
\[
\liminf_{r\to \infty} r^{-\varepsilon} \int_{B^n(r)} f^* \vol_N >0.
\]
\end{theorem}

The second result is a version of Rickman's Hunting Lemma \cite[Lemma 5.1]{RI3}. The following formulation is due to Bonk and Poggi-Corradini \cite{BP}.

\begin{lemma}{\cite[Lemma 2.1]{BP}}
\label{lem:hunting}
Let $\mu$ be an atomless Borel measure on $\R^n$ satisfying $\mu(\R^n)=\infty$ and $\mu(B)<\infty$ for every ball $B\subset \R^n$. Then there exists a constant $D=D(n)>1$ with the property that, for every $m\in \N$, there exists a ball $B\subset \R^n$ for which
\[
m\le \mu(2B)\le D\mu (B).
\]
\end{lemma}

Proposition \ref{prop:reduction} is an immediate consequence of these results.

\begin{proof}[Proof of Proposition \ref{prop:reduction}]
Since $N$ is $K$-quasiregularly elliptic, there exists a non-constant quasiregular map $f\colon \R^n \to N$ satisfying the distortion inequality \eqref{eq:QR} with constant $K$. By Theorem \ref{thm:fast-growth}, the Borel measure $\mu$ on $\R^n$ defined by $\mu(E)=\int_E f^*\vol_N$ satisfies the assumptions of Lemma \ref{lem:hunting}. Hence, there exist a constant $D=D(n)>1$ and a sequence of balls $(B^n(a_m,r_m))$ for which
\[
m\le \int_{B^n(a_m,2r_m)} f^*\vol_N \le D\int_{B^n(a_m,r_m)} f^*\vol_N.
\]
Now the maps $f_m\colon B_2^n\to N$, $x\mapsto f(r_mx+a_m)$, form a sequence in $\cF_{K,D}(N)$ satisfying $A(f_m)\to \infty$. The claim follows.
\end{proof}

\section{Sobolev algebra $\cW^*$ and normalized pull-backs}
\label{sec:prelims}

We begin by recalling briefly the necessary preliminaries on Sobolev differential forms. A $(k+1)$-form $d\omega \in L_{\loc}^1(B_2^n;\bigwedge^{k+1} \R^n)$ is the \emph{weak exterior derivative} of a $k$-form $\omega \in L_{\loc}^1(B_2^n;\bigwedge^k \R^n)$ if
\[
\int_{B_2^n} \omega \wedge d\varphi = (-1)^{k+1} \int_{B_2^n} d\omega \wedge \varphi
\]
for every $\varphi \in \Omega_0^{n-k-1}(B_2^n)$. By \cite[Lemma 3.6]{Iwaniec-Martin}, the exterior derivative commutes with the pullback of a quasiregular map. More precisely, let $N$ be a closed, connected, and oriented Riemannian $n$-manifold and let $f\colon B_2^n\to N$ be a quasiregular map. Then, for $\alpha \in \Omega^k(N)$, the form $f^*(d\alpha) \in L^\frac{n}{k+1}_{\loc}(B_2^n; \bigwedge^{k+1} \R^n)$ is the weak exterior derivative of $f^*\alpha \in L^\frac{n}{k}_{\loc}(B_2^n; \bigwedge^k \R^n)$, i.e., $d(f^*\alpha)=f^*(d\alpha)$; here the integrability of $f^* \alpha$ and $f^*(d\alpha)$ follow from the point-wise inequality
\[
\abs{f^* \beta} \le (\abs{\beta} \circ f) \norm{Df}^\ell
\]
for $\beta \in \Omega^\ell(N)$.

Let $W^{d,\frac{n}{k}}(B_2^n;\bigwedge^k \R^n)$ denote the Sobolev space of forms $\omega$ belonging to $L^\frac{n}{k}(B_2^n;\bigwedge^k \R^n)$ having a weak exterior derivative $d\omega \in L^\frac{n}{k}(B_2^n;\bigwedge^{k+1} \R^n)$. Suppose now that $f\colon B_2^n\to N$ is a quasiregular map with integrable Jacobian, that is, $J_f\in L^1(B_2^n)$. Then, for a closed form $\alpha \in \Omega^k(N)$, the pull-back $f^* \alpha$ belongs to $W^{d,\frac{n}{k}}(B_2^n;\bigwedge^k \R^n)$ since $f^* \alpha$ is weakly closed.

Let $f^! \colon \Omega^*(N)\to \oplus_{k=0}^n L^\frac{n}{k}(B_2^n;\bigwedge^k \R^n)$ be the linear map
\[
\Omega^k(N)\ni \alpha \mapsto A(f)^{-\frac{k}{n}}f^* \alpha,
\]
where $A(f)^{-\frac{k}{n}}=1$ for $k=0$; recall that
\[
A(f) = \int_{B^n} f^* \vol_N.
\]

It is easily seen that $f^!$ commutes with the exterior product and that $f^! \circ h = f^\#$. We also note that $f^!(\vol_N)$ is the normalized measure $\nu_f$ of $f$.

Next, we record an elementary lemma stating that, if $f\in \cF_{K,D}(N)$, then for $k$-forms the operator $f^!$ is bounded from $\Omega^k(N)\to L^\frac{n}{k}(B_2^n; \bigwedge^k \R^n)$ and the operator norm of $f^!$ is bounded by $K$ and $D$.

\begin{lemma}
\label{lem:normalized-is-bounded}
Let $N$ be a closed, connected, and oriented Riemannian $n$-manifold, $n\ge 2$, and let $f\in \cF_{K,D}(N)$. Then
\[
\norm{f^!(\alpha)}_{\frac{n}{k},B_2^n} \le DK\norm{\alpha}_\infty
\]
for every $\alpha \in \Omega^k(N)$ and for $k=0,\ldots,n$.
\end{lemma}

\begin{proof}
The claim holds trivially for $\alpha \in \Omega^0(N)$ as $\abs{f^!(\alpha)}(x)=\abs{\alpha(f(x))}$ for each $x\in B_2^n$. Hence, let $0<k\le n$ and $\alpha \in \Omega^k(N)$. Then
\begin{align*}
\norm{f^!(\alpha)}_{\frac{n}{k},B_2^n} &= A(f)^{-\frac{k}{n}} \left( \int_{B_2^n} \abs{f^* \alpha}^\frac{n}{k} \right)^\frac{k}{n} \\
&\le \norm{\alpha}_\infty A(f)^{-\frac{k}{n}} \left( \int_{B_2^n} \norm{Df}^n \right)^\frac{k}{n} \\
&\le \norm{\alpha}_\infty K^\frac{k}{n} A(f)^{-\frac{k}{n}} \left( \int_{B_2^n} f^*\vol_N \right)^\frac{k}{n} \\
&\le \norm{\alpha}_\infty K^\frac{k}{n} D^\frac{k}{n} \le DK\norm{\alpha}_\infty.
\end{align*}
This concludes the proof.
\end{proof}

We finish this section by showing that limits of exact forms are negligible under sequences $(f_j^!)$, if $(f_j)$ is a sequence in $\cF_{K,D}(N)$ satisfying $A(f_j)\to \infty$; see also \cite{PR}. For a similar observation, based on integration by parts, in the context of entire maps and in quasiregular dynamics see \cite{Pankka-GAFA}, \cite{Okuyama-Pankka}, \cite{Kangasniemi-Adv}, and \cite{Heikkila}.

\begin{lemma}
\label{lem:limit-of-exact}
Let $N$ be a closed, connected, and oriented Riemannian $n$-manifold, $n\ge 2$, and let $(f_j)$ be a sequence in $\cF_{K,D}(N)$ satisfying $A(f_j)\to \infty$. Then
\[
\lim_{j\to \infty} \int_{B_2^n} \varphi \wedge f_j^!(d\alpha) = 0
\]
for $\alpha \in \Omega^k(N)$, $\varphi \in \Omega_0^{n-k-1}(B_2^n)$, and $k=0,\ldots,n$.
\end{lemma}

\begin{proof}
We may assume that $0\le k<n$. Let $\alpha \in \Omega^k(N)$ and let $\varphi \in \Omega_0^{n-k-1}(B_2^n)$. By Hölder's inequality and Lemma \ref{lem:normalized-is-bounded}, we obtain
\begin{align*}
&\abs{\int_{B_2^n} \varphi \wedge f_j^!(d\alpha)} \\
&\quad = A(f_j)^{-\frac{k+1}{n}} \abs{\int_{B_2^n} \varphi \wedge f_j^*(d\alpha)} = A(f_j)^{-\frac{k+1}{n}} \abs{\int_{B_2^n} d\varphi \wedge f_j^*\alpha} \\
&\quad = A(f_j)^{-\frac{1}{n}} \abs{\int_{B_2^n} d\varphi \wedge f_j^!(\alpha)} \le \frac{C(n)}{A(f_j)^\frac{1}{n}} \int_{B_2^n} \abs{d\varphi} \abs{f_j^!(\alpha)} \\
&\quad \le \frac{C(n)}{A(f_j)^\frac{1}{n}} \norm{d\varphi}_{\frac{n}{n-k},B_2^n} \norm{f_j^!(\alpha)}_{\frac{n}{k},B_2^n} \\
&\quad \le \frac{C(n)}{A(f_j)^\frac{1}{n}} \norm{d\varphi}_{\frac{n}{n-k},B_2^n} DK\norm{\alpha}_\infty.
\end{align*}
The claim follows.
\end{proof}

\section{Limits of normalized pull-backs are Sobolev-Poincaré limits}
\label{sec:potential}


We begin by recalling the Poincaré operator of Iwaniec and Lutoborski \cite[Section 4]{IL}: \emph{There exists a graded compact linear operator
\[
T\colon \bigoplus_{k=1}^{n-1} L^\frac{n}{k}(B_2^n;{\bigwedge}^k \R^n) \to \bigoplus_{k=1}^{n-1} L^\frac{n}{k}(B_2^n;{\bigwedge}^{k-1} \R^n)
\]
satisfying $\mathrm{id}=dT+Td$. In particular, $dT(\omega)=\omega$ for every weakly closed form $\omega \in \oplus_{k=1}^{n-1} L^\frac{n}{k}(B_2^n;\bigwedge^k \R^n)$}.

The following lemma is a reformulation of \cite[Lemma 4.1]{PR}. We recall the proof for the reader's convenience. Recall also that, for a quasiregular map $f\colon B_2^n \to N$,
\[
f^\# \colon H_{\dR}^*(N)\to \cW^*(B_2^n), \; H_{\dR}^k(N) \ni c \mapsto A(f)^{-\frac{k}{n}} f^*(h_c) = (f^! \circ h)(c).
\]

\begin{lemma}
\label{lem:weak-is-exact-poincare}
Let $N$ be a closed, connected, and oriented Riemannian $n$-manifold, $n\ge 2$. Let $(f_j)$ be a sequence in $\cF_{K,D}(N)$ for which $f_j^\# \weakto L$, where
\[
f_j^\#, L \colon \bigoplus_{k=1}^{n-1} H^k_{\dR}(N) \to \bigoplus_{k=1}^{n-1} L^\frac{n}{k}(B_2^n; {\bigwedge}^k \R^n).
\]
Then there exist a subsequence $(f_{j_i})$ and a graded linear operator
\[
\widehat{L} \colon \bigoplus_{k=1}^{n-1} H_{\dR}^k(N) \to \bigoplus_{k=1}^{n-1} W^{d,\frac{n}{k}}(B_2^n;{\bigwedge}^{k-1} \R^n)
\]
satisfying $T(f_{j_i}^\#)\to \widehat{L}$. In particular, $d\widehat{L}=L$.
\end{lemma}

\begin{proof}
Let $c_1,\ldots,c_m$ be a basis of $\oplus_{k=1}^{n-1} H_{\dR}^k(N)$. Since
\[
\norm{f_j^\#(c_\ell)}_{\frac{n}{k_\ell},B_2^n} \le DK\norm{h(c_\ell)}_\infty,
\]
where $c_\ell \in H^{k_\ell}(N)$, for each $\ell=1,\ldots,m$ by Lemma \ref{lem:normalized-is-bounded} and $T$ is a compact operator, there exist a subsequence $(f_{j_i})$ and forms $\tau_1,\ldots,\tau_m \in \oplus_{k=1}^{n-1} L^\frac{n}{k}(B_2^n;\bigwedge^{k-1} \R^n)$ satisfying $T(f_{j_i}^\#(c_\ell))\to \tau_\ell$. Since $dT(f_j^\#(c_\ell))=f_j^\#(c_\ell)$, we have
\begin{align*}
\int_{B_2^n} \tau_\ell \wedge d\varphi &= \lim_{i\to \infty} \int_{B_2^n} T(f_{j_i}^\#(c_\ell)) \wedge d\varphi = (-1)^{k_\ell} \lim_{i\to \infty} \int_{B_2^n} f_j^\#(c_\ell) \wedge \varphi \\
&= (-1)^{k_\ell} \int_{B_2^n} L(c_\ell) \wedge \varphi
\end{align*}
for every $\varphi \in \Omega_0^{n-k_\ell}(B_2^n)$. Hence $d\tau_\ell=L(c_\ell)$. We define $\widehat{L}$ to be the linear map defined by $c_\ell \mapsto \tau_\ell$. The claim follows.
\end{proof}

Motivated by Lemma \ref{lem:weak-is-exact-poincare} we introduce the following notion. Let $N$ be a closed, connected, and oriented Riemannian $n$-manifold, $n\ge 2$, and let $(f_j)$ be a sequence in $\cF_{K,D}(N)$. We say that a graded linear operator
\[
\widehat{L} \colon \bigoplus_{k=1}^{n-1} H_{\dR}^k(N) \to \bigoplus_{k=1}^{n-1} W^{d,\frac{n}{k}}(B_2^n;{\bigwedge}^{k-1} \R^n)
\]
is a \emph{Sobolev--Poincaré limit of $(f_j^\#)$} if $d\widehat{L}$ is the weak limit of the operators $f_j^\#$ and $\widehat{L}$ is the limit of the operators $T(f_j^\#)$.

The weak exterior derivative of a Sobolev--Poincaré limit commutes with the wedge product in a weak sense; see also \cite[Lemma 4.2]{PR}. For the statement, we recall that a sequence $(\mu_j)$ of Radon measures on $B_2^n$ \emph{converges vaguely to a Radon measure $\mu$ on $B_2^n$} if
\[
\lim_{j\to \infty} \int_{B_2^n} \varphi \, \mathrm{d}\mu_j = \int_{B_2^n} \varphi \, \mathrm{d}\mu
\]
for every $\varphi \in C_0(B_2^n)$.

\begin{proposition}
\label{prop:weak-commutation}
Let $N$ be a closed, connected, and oriented Riemannian $n$-manifold, $n\ge 2$. Let $(f_j)$ be a sequence in $\cF_{K,D}(N)$ and let the operators $(f_j^\#)$ have a Sobolev--Poincaré limit $\widehat{L}$. Let also $k_1,k_2 \in \{1,\ldots,n-1\}$ satisfy $k_1+k_2\le n$. Then, for every $c_1 \in H_{\dR}^{k_1}(N)$ and $c_2 \in H_{\dR}^{k_2}(N)$, we have $f_j^\#(c_1) \wedge f_j^\#(c_2) \weakto d\widehat{L}(c_1) \wedge d\widehat{L}(c_2)$. Here, the weak convergence is the usual weak convergence in $L^\frac{n}{k_1+k_2}(B_2^n;\bigwedge^{k_1+k_2} \R^n)$ if $k_1+k_2<n$ and the vague convergence of measures if $k_1+k_2=n$.
\end{proposition}

\begin{proof}
Let $c_i \in H_{\dR}^{k_i}(N)$ for $i=1,2$. By Lemma \ref{lem:normalized-is-bounded}, the sequence $(f_j^\#(c_1) \wedge f_j^\#(c_2))=(f_j^!(h(c_1)\wedge h(c_2))$ is bounded in $L^\frac{n}{k_1+k_2}(B_2^n;\bigwedge^{k_1+k_2} \R^n)$. Moreover, $\Omega_0^{n-k_1-k_2}(B_2^n)$ is dense in $L^\frac{n}{n-k_1-k_2}(B_2^n;\bigwedge^{n-k_1-k_2} \R^n)$ with respect to the $L^\frac{n}{n-k_1-k_2}$-norm if $k_1+k_2<n$ and $C_0^\infty(B_2^n)$ is dense in $C_0(B_2^n)$ with respect to the $L^\infty$-norm. Thus, it suffices to prove that
\[
\lim_{j\to \infty} \int_{B_2^n} \varphi \wedge f_j^\#(c_1) \wedge f_j^\#(c_2) = \int_{B_2^n} \varphi \wedge d\widehat{L}(c_1) \wedge d\widehat{L}(c_2)
\]
for every $\varphi \in \Omega_0^{n-k_1-k_2}(B_2^n)$.

Let $\varphi \in \Omega_0^{n-k_1-k_2}(B_2^n)$. By the triangle inequality, it suffices to estimate the terms
\[
\abs{ \int_{B_2^n} \varphi \wedge (d\widehat{L}(c_1) - f_j^\#(c_1)) \wedge d\widehat{L}(c_2) }
\]
and
\[
\abs{ \int_{B_2^n} \varphi \wedge f_j^\#(c_1) \wedge (d\widehat{L}(c_2) - f_j^\#(c_2)) }.
\]
Since
\[
\varphi \wedge d\widehat{L}(c_2) \in L^\frac{n}{k_2}(B_2^n;{\bigwedge}^{n-k_1} \R^n) \subset L^\frac{n}{n-k_1}(B_2^n;{\bigwedge}^{n-k_1} \R^n)
\]
and $f_j^\#(c_1) \weakto d\widehat{L}(c_1)$ in $L^\frac{n}{k_1}(B_2^n;\bigwedge^{k_1} \R^n)$, the first term converges to zero. For the second term, since $dT(f_j^\#(c_2))=f_j^\#(c_2)$, we have the estimate
\begin{align*}
&\abs{ \int_{B_2^n} \varphi \wedge f_j^\#(c_1) \wedge (d\widehat{L}(c_2) - f_j^\#(c_2)) } \\
	&\quad = \abs{ \int_{B_2^n} \varphi \wedge f_j^\#(c_1) \wedge d(\widehat{L}(c_2) - T(f_j^\#(c_2))) } \\
	&\quad = \abs{ \int_{B_2^n} d\varphi \wedge f_j^\#(c_1) \wedge (\widehat{L}(c_2) - T(f_j^\#(c_2)) } \\
	&\quad \le C(n,\varphi) \int_{B_2^n} \abs{f_j^\#(c_1)} \abs{\widehat{L}(c_2) - T(f_j^\#(c_2))}.
\end{align*}
By Hölder's inequality and Lemma \ref{lem:normalized-is-bounded}, we further have that
\begin{align*}
&\int_{B_2^n} \abs{f_j^\#(c_1)} \abs{\widehat{L}(c_2) - T(f_j^\#(c_2)} \\
	&\quad \le \norm{f_j^\#(c_1)}_{\frac{n}{n-k_2},B_2^n} \norm{\widehat{L}(c_2) - T(f_j^\#(c_2))}_{\frac{n}{k_2},B_2^n} \\
	&\quad \le C(n,k_1,k_2) \norm{f_j^\#(c_1)}_{\frac{n}{k_1},B_2^n} \norm{\widehat{L}(c_2) - T(f_j^\#(c_2))}_{\frac{n}{k_2},B_2^n} \\	
	&\quad \le C(n,k_1,k_2,K,D,c_1) \norm{\widehat{L}(c_2) - T(f_j^\#(c_2))}_{\frac{n}{k_2},B_2^n}.
\end{align*}
Since $T(f_j^\#(c_2)) \to \widehat{L}(c_2)$ in $L^\frac{n}{k_2}(B_2^n;\bigwedge^{k_2} \R^n)$, the claim follows.
\end{proof}

Next, we prove that the exterior derivative of a Sobolev--Poincaré limit commutes with the wedge product provided that the normalizations grow without bound.

\begin{lemma}
\label{lem:commutation}
Let $N$ be a closed, connected, and oriented Riemannian $n$-manifold, $n\ge 2$. Let $(f_j)$ be a sequence in $\cF_{K,D}(N)$ for which $A(f_j)\to \infty$ and let the operators $(f_j^\#)$ have a Sobolev--Poincaré limit $\widehat{L}$. Then
\[
d\widehat{L}(c_1 \wedge c_2) = d\widehat{L}(c_1) \wedge d\widehat{L}(c_2)
\]
for every $c_1 \in H_{\dR}^{k_1}(N)$ and $c_2 \in H_{\dR}^{k_2}(N)$ with $k_1,k_2 \in \{1,\ldots,n-1\}$ satisfying $k_1+k_2<n$.
\end{lemma}

\begin{proof}
Let $k_1,k_2 \in \{1,\ldots,n-1\}$ satisfy $k_1+k_2<n$. Let $c_1 \in H_{\dR}^{k_1}(N)$ and $c_2 \in H_{\dR}^{k_2}(N)$. Then there exists a form $\tau \in \Omega^{k_1+k_2-1}(N)$ for which $h(c_1 \wedge c_2)=h(c_1)\wedge h(c_2)+d\tau$. Since $f_j^\#(c_1 \wedge c_2)\weakto d\widehat{L}(c_1 \wedge c_2)$ in $L^\frac{n}{k_1+k_2}(B_2^n;\bigwedge^{k_1+k_2} \R^n)$, we have
\begin{align*}
\int_{B_2^n} \varphi \wedge d\widehat{L}(c_1 \wedge c_2) &= \lim_{j\to \infty} \int_{B_2^n} \varphi \wedge f_j^\#(c_1 \wedge c_2) \\
&= \lim_{j\to \infty} \int_{B_2^n} \varphi \wedge (f_j^\#(c_1) \wedge f_j^\#(c_2) + f_j^!(d\tau)) \\
&= \int_{B_2^n} \varphi \wedge d\widehat{L}(c_1) \wedge d\widehat{L}(c_2)
\end{align*}
for every $\varphi \in \Omega_0^{n-k_1-k_2}(B_2^n)$ by Lemma \ref{lem:limit-of-exact} and Proposition \ref{prop:weak-commutation}. The claim follows.
\end{proof}

The following result shows that, if the target manifold $N$ is not rational homology sphere and $(f_j)$ is a sequence in $\cF_{K,D}(N)$ satisfying $A(f_j)\to \infty$, then a Sobolev--Poincaré limit $\widehat{L}$ of the operators $(f_j^\#)$ yields a vague limit for the normalized measures $\nu_{f_j}$. Furthermore, the limiting measure is an absolutely continuous measure; see Sections \ref{sec:example} and \ref{sec:measures} for further discussion on the vague convergence of normalized measures.

\begin{theorem}
\label{thm:limit-measure}
Let $N$ be a closed, connected, and oriented Riemannian $n$-manifold, $n\ge 2$, which is not a rational homology sphere. Let $(f_j)$ be a sequence in $\cF_{K,D}(N)$ for which $A(f_j)\to \infty$ and let the operators $(f_j^\#)$ have a Sobolev--Poincaré limit $\widehat{L}$. Then there exists a non-negative form $\nu_{d\widehat{L}} \in L^1(B_2^n;\bigwedge^n \R^n)$ for which $\nu_{f_j} \weakto \nu_{d\widehat{L}}$ vaguely as measures.
\end{theorem}

\begin{proof}
Since $N$ is not a rational homology sphere, there exist, by the Poincaré duality, an index $k\in \{1,\ldots,n-1\}$ and forms $c\in H_{\dR}^k(N)$ and $c'\in H_{\dR}^{n-k}(N)$ for which $\int_N c\wedge c'\ne 0$. Then
\[
\nu_{d\widehat{L}} = \frac{\vol(N)}{\int_N c\wedge c'} d\widehat{L}(c) \wedge d\widehat{L}(c')
\]
is a well-defined $n$-form in $L^1(B_2^n;\bigwedge^n \R^n)$, which we consider also as a measure.

Let $\tau \in \Omega^{n-1}(N)$ be a form satisfying
\[
\vol(N) h(c)\wedge h(c') = \left( \int_N c\wedge c'\right) (\vol_N+d\tau).
\]
Since $\norm{f_j^!(\vol_N)}_{1,B_2^n}\le D$ and $C_0^\infty(B_2^n)$ is dense in $C_0(B_2^n)$ with respect to the $L^\infty$-norm, it suffices to prove that
\[
\lim_{j\to \infty} \int_{B_2^n} \varphi \nu_{f_j} = \int_{B_2^n} \varphi \nu_{d\widehat{L}}
\]
for every $\varphi \in C_0^\infty(B_2^n)$. By Proposition \ref{prop:weak-commutation} and Lemma \ref{lem:limit-of-exact}, we obtain
\begin{align*}
\int_{B_2^n} \varphi \nu_{d\widehat{L}} &= \frac{\vol(N)}{\int_N c\wedge c'} \lim_{j\to \infty} \int_{B_2^n} \varphi f_j^\#(c) \wedge f_j^\#(c') \\
&= \frac{\vol(N)}{\int_N c\wedge c'} \lim_{j\to \infty} \int_{B_2^n} \varphi f_j^!(h(c)) \wedge f_j^!(h(c')) \\
&= \lim_{j\to \infty} \int_{B_2^n} \varphi f_j^!(\vol_N + d\tau) \\
&= \lim_{j\to \infty} \int_{B_2^n} \varphi \nu_{f_j}
\end{align*}
for every $\varphi \in C_0^\infty(B_2^n)$. Since the forms $\nu_{f_j}$ are non-negative, so is $\nu_{d\widehat{L}}$.
\end{proof}

We finish this section by showing that, when restricted to the closed unit ball $\bar B^n$, the vague convergence of measures in Theorem \ref{thm:limit-measure} improves into weak convergence of probability measures.

\begin{corollary}
\label{cor:probability-measure}
Let $N$ be a closed, connected, and oriented Riemannian $n$-manifold, $n\ge 2$, which is not a rational homology sphere. Let $(f_j)$ be a sequence in $\cF_{K,D}(N)$ for which $A(f_j)\to \infty$ and let the operators $(f_j^\#)$ have a Sobolev--Poincaré limit $\widehat{L}$. Let $\nu_{d\widehat{L}} \in L^1(B_2^n;\bigwedge^n \R^n)$ be the limit measure in Theorem \ref{thm:limit-measure}. Then $\nu_{f_j} \llcorner \bar B^n \weakto \nu_{d\widehat{L}} \llcorner \bar B^n$. In particular, $\nu_{d\widehat{L}}(B^n)=1$.
\end{corollary}

\begin{proof}
Let $\zeta \in C(\bar B^n)$ be a non-negative function. Let $\eta \in C_b(B_2^n)$ be a non-negative bounded function satisfying $\eta|_{\bar B^n}=\zeta$. Then $\eta \nu_{f_j} \weakto \eta \nu_{d\widehat{L}}$ vaguely as measures.

Let $0<r<1<R<2$. Let $\psi_r \in C_0(B^n)$ and $\psi_R \in C_0(B^n(0,R))$ satisfy $\chi_{B^n(0,r)} \le \psi_r \le 1$ and $\chi_{B^n} \le \psi_R \le 1$. Since $\eta \nu_{d\widehat{L}}$ is non-negative, we obtain
\[
\int_{\bar B^n(0,r)} \eta \nu_{d\widehat{L}} \le \int_{B_2^n} \psi_r \eta \nu_{d\widehat{L}} = \lim_{j\to \infty} \int_{B_2^n} \psi_r \eta \nu_{f_j} \le \liminf_{j\to \infty} \int_{B^n} \eta \nu_{f_j}
\]
and
\[
\int_{B^n(0,R)} \eta \nu_{d\widehat{L}} \ge \int_{B_2^n} \psi_R \eta \nu_{d\widehat{L}} = \lim_{j\to \infty} \int_{B_2^n} \psi_R \eta \nu_{f_j} \ge \limsup_{j\to \infty} \int_{B^n} \eta \nu_{f_j}.
\]
Since $\nu_{d\widehat{L}} \in L^1(B_2^n;\bigwedge^n \R^n)$,
\[
\liminf_{j\to \infty} \eta \nu_{f_j}(B^n) \ge \eta \nu_{d\widehat{L}}(B^n) = \eta \nu_{d\widehat{L}}(\bar B^n) \ge \limsup_{j\to \infty} \eta \nu_{f_j}(B^n).
\]
Thus
\begin{equation}
\label{eq:weak-convergence-for-zeta}
\lim_{j\to \infty} \int_{\bar B^n} \zeta \nu_{f_j} = \int_{\bar B^n} \zeta \nu_{d\widehat{L}}.    
\end{equation}
Equality \eqref{eq:weak-convergence-for-zeta} is then obtained for an arbitrary function $\zeta \in C(\bar B^n)$ by writing $\zeta = \max(\zeta,0) - \max(-\zeta,0)$.

Especially, since $\nu_{d\widehat{L}} \in L^1(B_2^n;\bigwedge^n \R^n)$, we have, for $\zeta=1$, that
\[
\int_{B^n} \nu_{d\widehat{L}} = \int_{\bar B^n} \nu_{d\widehat{L}} = \lim_{j\to \infty} \int_{\bar B^n} \nu_{f_j} = \lim_{j\to \infty} \int_{B^n} \nu_{f_j} = 1.
\]
\end{proof}

\section{Weak exterior derivatives of Sobolev--Poincaré limits extend to algebra homomorphisms}
\label{sec:extension}

In this section, we prove that, if $N$ is not a rational homology sphere and $\widehat{L}$ is a Sobolev--Poincaré limit induced by a sequence $(f_j)$ in $\cF_{K,D}(N)$ satisfying $A(f_j)\to \infty$, then the map $\Phi_{d\widehat{L}} \colon H_{\dR}^*(N) \to \cW^*(B_2^n)$ defined by
\[
H_{\dR}^k(N) \ni c \mapsto \left\{ \begin{array}{ll}
\text{constant function } h_c, & \text{for } k=0 \\
d\widehat{L}(c), & \text{for } 1\le k\le n-1 \\
\vol(N)^{-1} (\int_N c) \nu_{d\widehat{L}}, & \text{for } k =n
\end{array}\right.
\]
is a graded algebra homomorphism; here $\nu_{d\widehat{L}}$ is the limit $n$-form in Theorem \ref{thm:limit-measure}.

\begin{proposition}
\label{prop:extends-to-homomorphism}
Let $N$ be a closed, connected, and oriented Riemannian $n$-manifold, $n\ge 2$, which is not a rational homology sphere. Let $(f_j)$ be a sequence in $\cF_{K,D}(N)$ for which $A(f_j)\to \infty$ and let the operators $(f_j^\#)$ have a Sobolev--Poincaré limit $\widehat{L}$. Then $\Phi_{d\widehat{L}}$ is a graded algebra homomorphism.
\end{proposition}

\begin{proof}
Clearly $\Phi_{d\widehat{L}}$ is a graded linear map. It remains to prove that $\Phi_{d\widehat{L}}$ is an algebra homomorphism. Let $c_1 \in H_{\dR}^{k_1}(N)$ and $c_2 \in H_{\dR}^{k_2}(N)$, $k_1,k_2 \in \{0,\ldots,n\}$. Without loss of generality, we may assume that $k_1,k_2>0$ and that $k_1+k_2\le n$. If $k_1+k_2<n$, then, by Lemma \ref{lem:commutation}, we have that
\[
\Phi_{d\widehat{L}}(c_1 \wedge c_2) = d\widehat{L}(c_1 \wedge c_2) = d\widehat{L}(c_1) \wedge d\widehat{L}(c_2) = \Phi_{d\widehat{L}}(c_1) \wedge \Phi_{d\widehat{L}}(c_2).
\]
Suppose now that $k_1+k_2=n$ and let $\tau \in \Omega^{n-1}(N)$ be a form satisfying
\[
\vol(N) h(c_1) \wedge h(c_2) = \left(\int_N c_1 \wedge c_2\right)(\vol_N+d\tau).
\]
By Proposition \ref{prop:weak-commutation}, Lemma \ref{lem:limit-of-exact}, and Theorem \ref{thm:limit-measure}, we have
\begin{align*}
\int_{B_2^n} \varphi \Phi_{d\widehat{L}}(c_1) \wedge \Phi_{d\widehat{L}}(c_2) &= \int_{B_2^n} \varphi d\widehat{L}(c_1)\wedge d\widehat{L}(c_2) \\
&= \lim_{j\to \infty} \int_{B_2^n} \varphi f_j^\#(c_1)\wedge f_j^\#(c_2) \\
&= \frac{\int_N c_1 \wedge c_2}{\vol(N)} \lim_{j\to \infty} \int_{B_2^n} \varphi f_j^!(\vol_N + d\tau) \\
&= \frac{\int_N c_1 \wedge c_2}{\vol(N)} \lim_{j\to \infty} \int_{B_2^n} \varphi \nu_{f_j} \\
&= \frac{\int_N c_1 \wedge c_2}{\vol(N)} \int_{B_2^n} \varphi \nu_{d\widehat{L}} = \int_{B_2^n} \varphi \Phi_{d\widehat{L}}(c_1 \wedge c_2)
\end{align*}
for every $\varphi \in C_0^\infty(B_2^n)$. Thus $\Phi_{d\widehat{L}}(c_1)\wedge \Phi_{d\widehat{L}}(c_2)=\Phi_{d\widehat{L}}(c_1 \wedge c_2)$. This concludes the proof.
\end{proof}

\section{Limits of normalized pull-backs are algebra monomorphisms}
\label{sec:monomorphism}

In this section, we recall the statement of Theorem \ref{thm:limit} and give its proof.

\begin{named}{Theorem \ref{thm:limit}}
Let $N$ be a closed, connected, and oriented Riemannian $n$-manifold, $n\ge 2$. Let $(f_j)$ be a sequence in $\cF_{K,D}(N)$ for which $A(f_j) \to \infty$ and $f_j^\# \weakto L$, where
\[
f_j^\#, L \colon \bigoplus_{k=1}^{n-1} H^k_{\dR}(N) \to \bigoplus_{k=1}^{n-1} L^\frac{n}{k}(B_2^n; {\bigwedge}^k \R^n).
\]
Then $L$ extends to an embedding of graded algebras 
\[L\colon H^*_\dR(N) \to \cW^*(B_2^n)
\]
satisfying
\[
\int_{B^n} L([\vol_N])=1.
\]
\end{named}

\begin{proof}
We may assume that $N$ is not a rational homology sphere. By Lemma \ref{lem:weak-is-exact-poincare}, the operator $L$ is the weak exterior derivative of a Sobolev--Poincaré limit $\widehat{L}$ of the operators $(f_{j_i}^\#)$ for some subsequence $(f_{j_i})$. By Proposition \ref{prop:extends-to-homomorphism}, the map $\Phi_L = \Phi_{d\widehat{L}} \colon H_{\dR}^*(N) \to \cW^*(B_2^n)$ is an algebra homomorphism and $\Phi_L$ extends $L$ by definition. Moreover, by definition and Corollary \ref{cor:probability-measure}, we have
\[
\int_{B^n} \Phi_L([\vol_N]) = \nu_{d\widehat{L}}(B^n) = 1.
\]

It remains to prove that $\Phi_L$ is injective. Since $\Phi_L([\vol_N])\ne 0$, it suffices to show that $\Phi_L$ is injective on $H_{\dR}^k(N)$ for $1\le k\le n-1$. Let $c\in H_{\dR}^k(N)$ be non-zero. There exists $c' \in H_{\dR}^{n-k}(N)$ for which $\int_N c\wedge c'\ne 0$. Since
\[
\Phi_L(c) \wedge \Phi_L(c') = \Phi_L(c\wedge c') = \frac{\int_N c\wedge c'}{\vol(N)} \Phi_L([\vol_N]) \ne 0,
\]
we have that $\Phi_L(c)\ne 0$. This concludes the proof.
\end{proof}

\section{Proof of Theorem \ref{thm:sup}}
\label{sec:sup}

In this section, we recall the statement of Theorem \ref{thm:sup} and finish its proof.

\begin{named}{Theorem \ref{thm:sup}}
Let $N$ be a closed, connected, and oriented Riemannian $n$-manifold, $n\ge 2$, for which
\[
\sup_{f\in \cF_{K,D}(N)} A(f) = \infty.
\]
Then there exists an embedding of graded algebras $H^*_{\dR}(N) \to \bigwedge^* \R^n$.
\end{named}

\begin{proof}
Since $\sup_{f\in \cF_{K,D}(N)} A(f)=\infty$ and $\oplus_{k=1}^{n-1} H_{\dR}^k(N)$ is finite-dimensional, by Lemma \ref{lem:normalized-is-bounded} and the Banach--Alaoglu theorem, there exist a sequence $(f_j)$ in $\cF_{K,D}(N)$ for which $A(f_j)\to \infty$ and an operator
\[
L\colon \oplus_{k=1}^{n-1} H_{\dR}^k(N) \to \oplus_{k=1}^{n-1} L^\frac{n}{k}(B_2^n;{\bigwedge}^k \R^n),
\]
which is the weak limit of the sequence $(f_j^\#)$. By Theorem \ref{thm:limit}, the operator $L$ extends to an embedding of graded algebras $L\colon H_{\dR}^*(N)\to \cW^*(B_2^n)$ satisfying
\[
\int_{B^n} L([\vol_N]) = 1.
\]

Let $c_1,\ldots,c_m$ be a basis of $H_{\dR}^*(N)$. We fix point-wise Borel representatives for $L([\vol_N])$, $L(c_{\ell_1} \wedge c_{\ell_2})$, and $L(c_{\ell_1}) \wedge L(c_{\ell_2})$, where $\ell_1,\ell_2=1,\ldots,m$. Let
\[
E = \{ x\in B^n \colon (L([\vol_N]))(x)\ne 0 \}
\]
and
\[
E_{\ell_1,\ell_2} = \{ x\in B^n \colon (L(c_{\ell_1} \wedge c_{\ell_2}))(x) = (L(c_{\ell_1}))(x) \wedge (L(c_{\ell_2}))(x) \}
\]
for $\ell_1,\ell_2=1,\ldots,m$. Since the set $E$ has positive $n$-dimensional Lebesgue measure and the sets $E_{\ell_1,\ell_2}$ have full $n$-dimensional Lebesgue measure in $B^n$, the set $E' = \cap_{\ell_1,\ell_2=1}^m E\cap E_{\ell_1,\ell_2}$ has positive $n$-dimensional Lebesgue measure.

Let $x_0 \in E'$. Now the map
\[
\ccL \colon H_{\dR}^*(N)\to {\bigwedge}^* \R^n, \quad c\mapsto (L(c))(x_0),
\]
is a graded algebra homomorphism, which is injective on $H_{\dR}^n(N)$. If $N$ is a rational homology sphere, then $\ccL$ is trivially a monomorphism. If $N$ is not a rational homology sphere, the injectivity of $\ccL$ follows from the non-degeneracy of the pairing $H_{\dR}^*(N) \times H_{\dR}^{n-*}(N)\to \R$, $(c,c')\mapsto \int_N c\wedge c'$.
\end{proof}

\subsection{Remark on the use of harmonic representatives}
\label{subsec:remarks}

We defined operators $f^\# \colon H_{\dR}^*(N) \to \cW^*(B_2^n)$ as composition $f^\# = f^! \circ h$, where $h\colon H_{\dR}^*(N)\to \Omega^*(N)$ is the section of $\ker d\to H_{\dR}^*(N)$, $\alpha \mapsto [\alpha]$, assigning to each cohomology class its harmonic representative. In fact, we could have taken any section $\xi$ of $\ker d\to H_{\dR}^*(N)$, $\alpha \mapsto [\alpha]$, and defined a normalized $\xi$-pullback as composition $f^\xi = f^! \circ \xi$. The limits of normalized $\xi$-pullbacks are independent of the section $\xi$ in the following sense.

\begin{lemma}
\label{lem:section}
Let $N$ be a closed, connected, and oriented Riemannian $n$-manifold, $n\ge 2$, and let $(f_j)$ be a sequence in $\cF_{K,D}(N)$ for which $A(f_j)\to \infty$. Let $\xi\colon H_{\dR}^*(N)\to \Omega^*(N)$ be a section of $\ker d\to H_{\dR}^*(N)$, $\alpha \mapsto [\alpha]$, and let
\[
L,L' \colon \bigoplus_{k=1}^{n-1} H_{\dR}^k(N) \to \bigoplus_{k=1}^{n-1} L^\frac{n}{k}(B_2^n;{\bigwedge}^k \R^n)
\]
be weak limits of $f_j^\#$ and $f_j^\xi$, respectively. Then $L=L'$.
\end{lemma}

The proof is a simple modification of the proof of Lemma \ref{lem:weak-is-exact-poincare} since, for each $c\in H_{\dR}^*(N)$, the form $\xi(c)-h(c)$ is exact. We leave the details to the interested reader.

\subsection{Remark on the non-uniqueness of the embedding $H^*_{\dR}(N) \to \cW^*(B_2^n)$}
\label{subsec:unique}

Unless the manifold $N$ is a rational homology sphere, the embeddings of graded algebras $H^*_{\dR}(N)\to \cW^*(B_2^n)$ in Theorem \ref{thm:limit} are non-unique. Indeed, given a sequence $(f_j)$ for which the sequence $f_j^\# \weakto L$ in the sense of Theorem \ref{thm:limit} which admits a linear map $Q\in \mathrm{SO}_n(\R)$ satisfying $Q^* L\ne L$, then $(f_j \circ Q)$ is a sequence in $\cF_{K,D}(N)$ for which $(f_j\circ Q)^\# \weakto Q^* L\ne L$ in the sense of Theorem \ref{thm:limit}.

Even if we restrict ourselves to embeddings $H^*_{\dR}(N) \to \cW^*(B_2^n)$ associated to sequences $(f_j \colon B_2^n \to N)$, which stem from a single quasiregular map $\R^n \to N$, the uniqueness of the limit is not guaranteed. For statements of positive and negative results, let $f\colon \R^n \to N$ be a non-constant quasiregular map and consider the family of all mappings $f_{a,r} \colon B_2^n \to N$, $x\mapsto f \circ T_{a,r}$, where $a\in \R^n$, $r>0$, and $T_{a,r} \colon \R^n \to \R^n$  is the mapping $x\mapsto r x + a$.

We begin with an observation that, for the Riemannian covering map $\R^n \to \bT^n$, the induced algebra homomorphism $H_{\dR}^*(\bT^n)\to \cW^*(B_2^n)$ is unique.

\begin{example}
Let $\pi \colon \R^n \to \bT^n$, $(x_1,\ldots, x_n) \mapsto (e^{ix_1},\ldots, e^{ix_n})$, be the standard Riemannian covering map and let $\theta_1,\ldots, \theta_n$ be the $1$-forms for which $\pi^*\theta_i = dx_i$ for each $i=1,\ldots, n$. Let also $L_{\pi} \colon H^*_\dR(\bT^n) \to \cW^*(B_2^n)$ the embedding $[\theta_{i_1} \wedge \cdots \wedge \theta_{i_k}] \to m_n(B^n)^{-\frac{k}{n}} dx_{i_1} \wedge \cdots dx_{i_k}$, where $m_n(B^n)=\int_{B^n} \vol_{\R^n}$.

Since $\pi$ is a Riemannian covering map, it is $1$-quasiregular. We also have that all the mappings $\pi_{a,r} = \pi \circ T_{a,r}|_{B_2^n} \colon B_2^n \to \bT^n$ are $1$-quasiregular. It is also easy to deduce that $\pi_{a,r}^\# = L_\pi$ for each $a\in \R^n$ and $r>0$. Indeed, it suffices to observe that   
\begin{align*}
\pi_{a,r}^!(\theta_{i_1}\wedge \cdots \wedge \theta_{i_k}) &= \frac{1}{A(\pi \circ T_{a,r})^{k/n}} T_{a,r}^* \circ \pi^*(\theta_{i_1}\wedge \cdots \wedge \theta_{i_k}) \\
&= \frac{1}{r^k m_n(B^n)^\frac{k}{n}} T_{a,r}^* (dx_{i_1} \wedge \cdots \wedge dx_{i_k}) \\
&= \frac{1}{m_n(B^n)^\frac{k}{n}} dx_{i_1}\wedge \cdots \wedge dx_{i_k}
\end{align*}
for each multi-index $(i_1,\ldots, i_k)$. 
\end{example}

We modify now the Riemannian covering mapping $\R^n \to \bT^n$ in a sequence of Euclidean balls to obtain a quasiregular map $\bR^n \to \bT^n$ for which the induced embedding $H_\dR^*(\bT^n) \to \cW^*(B_2^n)$ is non-unique.

\begin{example}
Let $Q\in \mathrm{SO}_n(\R)$ be a linear map, which is not the identity. Let also $\alpha \colon [0,1]\to \mathrm{SO}_n(\R)$ be a geodesic in $\mathrm{SO}_n(\R)$ from the identity to $Q$. We denote $Q_t = \alpha(t)$ for each $t\in [0,1]$. Let also $B_j = B^n(a_j, r_j) = B^n(2^j e_1,j)$ be a sequence of mutually disjoint Euclidean balls for which $r_j \to \infty$.

We define now $h_Q \colon \R^n \to \R^n$ as follows. In the complement of the union $\bigcup_j B_j$, the mapping $h_Q$ is the identity. For $x\in B^n(a_j,r_j/2)$ , we set $h_Q(x)=a_j + r_j Q((x-a_j)/r_j)$. Finally, if $x\in B^n(a_j,r_j)\setminus B^n(a_j,r_j/2)$, we set $h_Q(x) = a_j + r_j Q_{2-2|x-a_j|/r_j}((x - a_j)/r_j)$. Since $\alpha$ is a geodesic and each $Q_t$ is an isometry, the mapping $h_Q$ is quasiconformal. 

Let now $f \colon \R^n \to \bT^n$ be the quasiregular map $f = \pi \circ h_Q$. Then, the sequence $(f_{-je_1, j})$ yields the limiting embedding $L_\pi$ and the sequence $(f_{a_j,r_j/2})$ the limiting embedding $Q^* L_\pi \ne L_\pi$.
\end{example}

\section{Example of a singular limit of normalized measures of quasiregular mappings $B_2^n \to \bS^n$}
\label{sec:example}

In this subsection, we show that, for $n\ge 2$, there exist a constant $K=K(n)\ge 1$ and maps $f_j \in \cF_{K,3}(\bS^n)$ for which the normalized measures $\nu_{f_j}$ converge vaguely to the Borel measure $\nu$,
\[
E\mapsto \frac{1}{2} \cH^1(E\cap J),
\]
where $J=[-1,1]\times \{0\}^{n-1}$. Since $\nu$ is singular with respect to the $n$-dimensional Lebesgue measure, this example shows that Theorem \ref{thm:measures} does not hold for quasiregular mappings into $\bS^n$. For the construction of the maps $f_j$, see also \cite[Example 24]{Pankka-Duke}.

Let $\iota \colon \R^n \to \bS^n$ be the stereographic projection and let $F\colon \bS^n \to \bS^n$ be the restriction of the map $\R^{n-1} \times \C\to \R^{n-1} \times \C $, $(x,re^{i\theta})\mapsto (x,re^{i3\theta})$. We remark that $F(z)=z$ for points $z$ on the equator $\bS^{n-1}\subset \bS^n$ and that the map $F$ is $K$-quasiregular for $K=K(n)\ge 1$.

For $j\in \N$ and for $i=1,\ldots,j$, let $a_{ij}=(-1+(2i-1)/j,0,\ldots,0)\in \R^n$, $B_{ij}=B^n(a_{ij},1/j)$, and $U_j=\bigcup_{i=1}^j B_{ij}\subset B^n$. Let also $\rho_{ij} \colon \R^n \to \R^n$ be the mapping $x\mapsto j(x-a_{ij})$. Finally, let $\sigma_{ij}$ be the M\"obius transformation
of $\bS^n$ satisfying $\sigma_{ij} \circ \iota = \iota \circ \rho_{ij}$.

We define now $f_j \colon B_2^n \to \bS^n$ to be the mapping $f_j(x)=\iota(x)$ for $x\notin U_j$ and $f_j(x)=\sigma_{ij}^{-1}(F(\sigma_{ij}(\iota(x))))$ for $x\in B_{ij}$ and each $i=1,\ldots,j$. Each map $f_j$ is $K$-quasiregular and, in each ball $B_{ij}$, we have that
\[
\int_{B_{ij}} f_j^* \vol_N = \int_{\sigma_{ij}^{-1}(\bS_+^n)} 2 + \int_{\sigma_{ij}^{-1}(\bS_-^n)} 1  = \vol(\sigma_{ij}^{-1}(\bS_+^n)) + \vol(\bS^n),
\]
where $\bS_+^n$ and $\bS_-^n$ are the upper and lower hemispheres of $\bS^n$, respectively. Hence
\[
\int_{B_2^n} f_j^* \vol_N \le 3j \vol(\bS^n) \le 3\int_{B^n} f_j^* \vol_N.
\]
Thus $f_j \in \cF_{K,3}(\bS^n)$.

It remains to show that the normalized measures $\nu_{f_j}$ converge vaguely to $\nu$. Let $\psi \in C_0(B_2^n)$. Denote
\[
\mu_{f_j} := \left( \int_{U_j} f_j^* \vol_N \right)^{-1} f_j^* \vol_N \quad \text{and} \quad \psi_{ij} := \int_{J\cap B_{ij}} \psi \dn \nu.
\]
Since
\[
\frac{1}{A(f_j)} \int_{B_2^n \setminus U_j} f_j^* \vol_N \le \frac{1}{j}
\]
for each $j$, it suffices to show that
\[
\lim_{j\to \infty} \int_{U_j} \psi \mu_{f_j} = \int_{B_2^n} \psi \dn \nu.
\]

Since $\sigma_{ij}^{-1}(\bS_+^n)=\iota(B^n(a_{ij},1/j))$ and $\iota$ is $\tilde L$-bilipschitz in $B^n$, where $\tilde L =\tilde L(n)$, we have that
\[
C'(n) \left( \frac{1}{\tilde L j} \right)^n + \vol(\bS^n) \le \int_{B_{ij}} f_j^* \vol_N \le C'(n) \left( \frac{\tilde L}{j} \right )^n + \vol(\bS^n)
\]
and thus
\[
\abs{ \int_{B_{ij}} \mu_{f_j} - \frac{1}{j} } \le C(n) \left( \frac{1}{j} \right)^{n+1}.
\]
Hence, we have the estimate
\begin{align*}
\abs{ \int_{U_j} \psi \mu_{f_j} - \int_{B_2^n} \psi \dn \nu } &= \abs{ \sum_{i=1}^j \int_{B_{ij}} \psi \mu_{f_j} - \sum_{i=1}^j \int_{J\cap B_{ij}} \psi \dn \nu }  \\
&\le \sum_{i=1}^j \abs{ \int_{B_{ij}} \psi \mu_{f_j} - \psi_{ij} } \\
&\le \sum_{i=1}^j \abs{ \int_{B_{ij}} (\psi-j\psi_{ij}) \mu_{f_j} } + \abs{ j\psi_{ij} \int_{B_{ij}} \mu_{f_j} - \psi_{ij} } \\
&\le \sum_{i=1}^j \sup_{B_{ij}} \abs{\psi-j\psi_{ij}} \int_{B_{ij}} \mu_{f_j} + \norm{\psi}_\infty \abs{ \int_{B_{ij}} \mu_{f_j} - \frac{1}{j} } \\
&\le \sum_{i=1}^j \sup_{B_{ij}} \abs{\psi-j\psi_{ij}} C(n) \frac{1}{j} + \norm{\psi}_\infty C(n) \left( \frac{1}{j} \right)^{n+1} \\
&\le \max_{1\le i\le j} \sup_{B_{ij}} \abs{\psi-j\psi_{ij}} C(n) + \norm{\psi}_\infty C(n) \left( \frac{1}{j} \right)^n.
\end{align*}
Since $\psi \in C_0(B_2^n)$, it holds that $\max_{1\le i\le j} \sup_{B_{ij}} \abs{\psi-j\psi_{ij}}\to 0$. Thus, we obtain that that the normalized measures $\nu_{f_j}$ converge vaguely to the singular measure $\nu$.

\section{Weak limits of normalized measures}
\label{sec:measures}

In this section, we prove the following result, which yields Corollary \ref{cor:sup} and Theorem \ref{thm:measures} stated in the introduction.

\begin{theorem}
\label{thm:bubbling-corollary}
Let $N$ be a closed, connected, and oriented Riemannian $n$-manifold, $n\ge 2$, which is not a rational homology sphere, i.e.~$H^k(N)\ne 0$ for some $k\in \{1,\ldots, n-1\}$. Then, up to passing to a subsequence, a sequence $(f_j)$ in $\cF_{K,D}(N)$ satisfying $\sup_j A(f_j)<\infty$ converges locally uniformly to a $K$-quasiregular map $f\colon B_2^n \to N$ in $\overline{\cF_{K,D}(N)}$.
\end{theorem}

Theorem \ref{thm:bubbling-corollary} implies Corollary \ref{cor:sup} almost immediately. 

\begin{named}{Corollary \ref{cor:sup}}
Let $N$ be a closed, connected, and oriented Riemannian $n$-manifold, $n\ge 2$. If $H^*_{\dR}(N)$ is not a subalgebra of $\bigwedge^* \R^n$, then $\overline{\cF_{K,D}(N)}$ is compact with respect to the topology of local uniform convergence. In particular, in this case, $\cF_ {K,D}(N)$ is a normal family.
\end{named}

\begin{proof}
Since the closure $\overline{\cF_{K,D}(N)}$ of quasiregular mappings in $\cF_{K,D}(N)$ in the topology of local uniform convergence consists of $\cF_{K,D}(N)$ and constant mappings $B_2^n \to N$, it suffices to show that a sequence in $\cF_{K,D}(N)$ has a subsequence converging either to a mapping in $\cF_{K,D}(N)$ or to a constant mapping.

Since $H_{\dR}^*(N)$ is not a subalgebra of $\bigwedge^* \R^n$, we have that 
\[
\sup_{f\in \cF_{K,D}(N)} A(f)<\infty
\]
by Theorem \ref{thm:sup} and that $N$ is not a rational homology sphere. Thus, by Theorem \ref{thm:bubbling-corollary}, each sequence $(f_j)$ in $\cF_{K,D}(N)$ has a subsequence $(f_{j_i})$ converging locally uniformly to a (possibly constant) quasiregular map $B_2^n \to N$ in $\overline{\cF_{K,D}(N)}$. 
We conclude that the family $\overline{\cF_{K,D}(N)}$ is compact.
\end{proof}

As a consequence of Theorem \ref{thm:bubbling-corollary}, we obtain also Theorem \ref{thm:measures}.

\begin{named}{Theorem \ref{thm:measures}}
Let $N$ be a closed, connected, and oriented Riemannian $n$-manifold, $n \ge 2$, which is not a rational homology sphere. 
Then a sequence $(f_j)$ in $\cF_{K,D}(N)$ has a subsequence $(f_{j_i})$ for which either $(f_{j_i})$ converges to a constant map or the normalized measures $\nu_{f_{j_i}}$ converge vaguely to a measure $\nu \ll \vol_{\R^n}$ in $B_2^n$.
\end{named}

\begin{proof}
If $\sup_j A(f_j)=\infty$, then there exists a subsequence $(f_{j_i})$ for which the operators $(f_{j_i}^\#)$ have a Sobolev--Poincaré limit $\widehat{L}$  and, by Theorem \ref{thm:limit-measure}, $\nu_{f_{j_i}} \weakto \nu_{d\widehat{L}}$ vaguely as measures, where $\nu_{d\widehat{L}} \ll \vol_{\R^n}$.

Suppose now that $\sup_j A(f_j)<\infty$. By Theorem \ref{thm:bubbling-corollary}, there exists a subsequence $(f_{j_i})$ converging locally uniformly to a map $f\in \overline{\cF_{K,D}(N)}$. If $f$ is non-constant, then, by the weak convergence of the Jacobians (see e.g. \cite[Lemma VI.8.8]{Rickman-book}), we have that $\nu_{f_{j_i}} \weakto \nu_f$ vaguely as measures, where $\nu_f  \ll \vol_{\R^n}$. 
\end{proof}

For the rest of this section, we will describe the idea of the proof of Theorem \ref{thm:bubbling-corollary}. This theorem is a particular case of the following version of the Gromov compactness theorem for quasiregular mappings; see \cite[Theorem 1.1]{Pankka-Souto}.

\begin{theorem}
\label{thm:bubbling}
Let $N$ be a closed, connected, and oriented Riemannian $n$-manifold, $n\ge 2$. Let $(f_j)$ be a sequence in $\cF_{K,D}(N)$ for which $\sup_j A(f_j)<\infty$. Then, up to passing to a subsequence, $(f_j)$ converges locally uniformly to a $K$-quasiregular map $f\colon X\to N$, for which $J_f \in L^1(X)$, on some nodal Riemannian $n$-manifold $X$. Furthermore, the main stratum of $X$ is diffeomorphic to $B_2^n$ and all of the bubbles of $X$ are conformal to $\bS^n$.
\end{theorem}

\newcommand{\Sing}{\mathrm{Sing}}
Recall that a (pure) nodal $n$-manifold $X$ is a topological space, which in the complement of a finite set $\Sing(X) \subset X$ is an $n$-manifold, each point in $\Sing(X)$ separates $X$ into two components, and the local structure of $X$ near a point $p\in \Sing(X)$ is modeled by two $n$-dimensional planes in general position in $\R^n \times \R^n$. The closure of each connected component of $X\setminus \Sing(X)$ is a called a \emph{stratum}. In Theorem \ref{thm:bubbling}, each strata is a smooth $n$-manifold and, more precisely, one stratum is diffeomorphic to $B_2^n$ and all other strata are diffeomorphic to $n$-spheres $\bS^n$. We refer to \cite[Section 4]{Pankka-Souto} for further discussion.

In \cite{Pankka-Souto} this convergence theorem is stated for $K$-quasiregular mappings $f_j \colon M\to N$ of fixed degree $\deg f_j$ between closed manifolds and it was shown that the limit map $f \colon X\to N$ has the same degree in the nodal sense; see \cite[Theorem 1.1]{Pankka-Souto} for a precise statement.

In Theorem \ref{thm:bubbling}, the condition on the degree is replaced by a uniform bound for the energy $A(f_j)$. The main difference is that a sequence $(f_j)$ in $\cF_{K,D}(N)$ may converge to a constant map and hence we may observe a loss of energy $A(f_j)$ at the limit. Indeed, it suffices to replace \cite[Proposition 3.1]{Pankka-Souto} by the following statement, whose proof is almost verbatim to the proof of \cite[Proposition 3.1]{Pankka-Souto}.

\begin{proposition}{\cite[Proposition 3.1]{Pankka-Souto}}
\label{prop:poles}
Let $N$ be a closed and oriented Riemannian $n$-manifold and let $(f_j)$ be a sequence in $\cF_{K,D}(N)$ satisfying $\sup_j A(f_j)<\infty$. Then, up to passing to a subsequence, there exists a finite set $P\subset B_2^n$ with the following properties:
\begin{itemize}
\item[\emph{(1)}] The measures $f_j^* \vol_N$ converge vaguely to a measure $\mu$ satisfying $\mu(B_2^n)\le D \cdot \sup_j A(f_j)$.
\item[\emph{2}] There exists a $K$-quasiregular map $f\colon B_2^n \to N$ with $f^* \vol_N= \mu$ on $B_2^n \setminus P$.
\item[\emph{3}] The sequence $(f_j)$ converges locally uniformly to $f$ in $B_2^n \setminus P$.
\item[\emph{4}] For each $p\in P$, we have $\mu(p)=d_p \cdot \vol(N)$ for some integer $d_p \ge 1$.
\end{itemize}
\end{proposition}

We continue now to the proof of Theorem \ref{thm:bubbling-corollary}.

\begin{proof}[Proof of Theorem \ref{thm:bubbling-corollary}]
Let $(f_j)$ be a sequence in $\cF_{K,D}(N)$ with $\sup_j A(f_j)<\infty$. Then, by Theorem \ref{thm:bubbling}, the sequence $(f_j)$ has a subsequence $(f_{j_i})$, which converges to a $K$-quasiregular mapping $f\colon X\to N$, where $X$ is a nodal $n$-manifold having $B_2^n$ as its main stratum and (possibly) $n$-spheres as other strata. 

Let $S \subset X$ be a stratum in $X$, which is an $n$-sphere. Then $f|_S \colon S\to N$ is a quasiregular map. Since $N$ is not a rational homology sphere, we have -- by the Poincar\'e duality -- that $\deg(f|_S \colon S\to N) = 0$. Thus $f|_S$ is constant. Hence $f$ is a constant map outside the main stratum of $X$. Thus $d_p = 0$ for all $p\in \Sing(X)$. We conclude that $(f_{j_i})$ converges locally uniformly to a $K$-quasiregular map $B_2^n \to N$. Moreover, we have
\[
\int_{B_2^n} f^* \vol_N \le \liminf_{j\to \infty} \int_{B_2^n} f_j^* \vol_N \le D\lim_{j\to \infty} \int_{B^n} f_j^* \vol_N = D\int_{B^n} f^* \vol_N
\]
by the weak convergence of the Jacobians. This concludes the proof.
\end{proof}

\appendix

\section{Donaldson--Freedman--Serre classification of closed $4$-manifolds}
\label{app:classification}

In this appendix, we recall the part of the topological classification of smooth, closed, simply connected, and oriented $4$-manifolds we use for Corollary \ref{cor:classification}. We also give upper bounds for the dimensions of positive and negative definite subspaces of the intersection form of a $4m$-manifold whose de Rham algebra embeds as a subalgebra of $\bigwedge^* \R^{4m}$.

We begin by recalling the definition of an intersection form. Let $N$ be a closed, connected, and oriented $4m$-manifold. The \emph{intersection form of $N$} is the bilinear form
\[
Q_N \colon H^{2m}(N) / \tor H^{2m}(N) \times H^{2m}(N) / \tor H^{2m}(N) \to \Z,
\]
\[
(u,v)\mapsto (u\smile v)([N]),
\]
where $[N]\in H_{2m}(N)$ is the fundamental class of $N$ associated to the orientation of $N$. Let $\beta_{2m}^+(N)$ and $\beta_{2m}^-(N)$ denote the maximal dimensions of vector subspaces of $H^{2m}(N;\R)$ in which $Q_N$ is positive and negative definite, respectively.

For $m=1$, Freedman \cite[Theorem 1.5]{FR} implies that smooth, closed, simply connected, and oriented $4$-manifolds are topologically classified by their intersection forms, i.e., two manifolds are homeomorphic if and only if they have the same intersection form. On the other hand, theorems of Donaldson \cite{DO} and Serre \cite[Section V.2.2]{SE} yield a list of symmetric, unimodular, and bilinear forms which contains every intersection form of a smooth, closed, simply connected, and oriented $4$-manifold. To avoid unnecessary technicalities, we do not recite the Donaldson--Serre list in detail, but observe that we obtain the following table of smooth, closed, simply connected, and oriented $4$-manifolds $N$ for which $\beta_{2m}^+(N) \le 3$ and $\beta_{2m}^-(N) \le 3$:

\vspace{0.25cm}
\begin{center}
\renewcommand*{\arraystretch}{1.4}
\begin{tabular}{|c|c|c|c|c|}
\hline
\diagbox{$\beta_2^-$}{$\beta_2^+$} & 0 & 1 & 2 & 3 \\
\hline
0 & $\bS^4$ & $\C P^2$ & $\#^2 \C P^2$ & $\#^3 \C P^2$ \\
\hline
1 & $\overline{\C P^2}$ & $\C P^2 \# \overline{\C P^2}$, & $\#^2 \C P^2 \# \overline{\C P^2}$ & $\#^3 \C P^2 \# \overline{\C P^2}$ \\
 & & $\bS^2 \times \bS^2$ & & \\
\hline
2 & $\#^2 \overline{\C P^2}$ & $\C P^2 \#^2 \overline{\C P^2}$ & $\#^2 \C P^2 \#^2 \overline{\C P^2}$, & $\#^3 \C P^2 \#^2 \overline{\C P^2}$ \\
 & & & $\#^2 \bS^2 \times \bS^2$ & \\
\hline
3 & $\#^3 \overline{\C P^2}$ & $\C P^2 \#^3 \overline{\C P^2}$ & $\#^2 \C P^2 \#^3 \overline{\C P^2}$ & $\#^3 \C P^2 \#^3 \overline{\C P^2}$, \\
 & & & & $\#^3 \bS^2 \times \bS^2$ \\
\hline
\end{tabular}
\end{center}
\vspace{0.25cm}

For the classification of closed simply connected quasiregularly elliptic $4$-manifolds $N$ it therefore suffices to observe that $\beta_{2m}^{\pm}(N) \le 3$. The following lemma records this simple algebraic consequence of Theorem \ref{thm:main}.

\begin{lemma}
\label{lem:bounds}
Let $N$ be a smooth, closed, connected, and oriented $4m$-manifold and let $\Phi \colon H^*(N,\R)\to \bigwedge^* \R^{4m}$ be an embedding of graded algebras. Then
\[
\beta_{2m}^{\pm}(N) \le \frac{1}{2} \binom{4m}{2m}.
\]
\end{lemma}

\begin{proof}
Let $V$ be a definite vector subspace of $Q_N$ in $H^{2m}(N;\R)$ and let $\ccR \colon H^*(N;\R)\to H_{\dR}^*(N)$ be the de Rham isomorphism. Note that $\ccR$ is an isomorphism of graded algebras; see e.g. \cite[Theorem III.3.1]{BR}.

Let $I\colon \bigwedge^{2m} \R^{4m} \to \R$ be the bilinear form $(\omega,\tau)\mapsto \int_{\R^{4m}} \omega \wedge \tau$. The bilinear form $I$ is either positive or negative definite in $\Phi(\ccR(V))$ since
\[
I(\Phi(\ccR(u)),\Phi(\ccR(u))) = \int_{\R^{4m}} \Phi(\ccR u\wedge \ccR u) = \frac{\int_{\R^{4m}} \Phi([\vol_N])}{\vol(N)} Q_N(u,u),
\]
where $\int_{\R^{4m}} \Phi([\vol_N])\ne 0$. Since a definite vector subspace of $I$ in $\bigwedge^{2m} \R^{4m}$ has dimension at most $\frac{1}{2} \binom{4m}{2m}$, we obtain that
\[
\dim V=\dim \Phi(\ccR(V))\le \frac{1}{2} \binom{4m}{2m}.
\]
The claim follows.
\end{proof}

\section{Quasiregularity of piecewise linear maps}
\label{app:PL}

In this appendix, we discuss the quasiregularity of piecewise linear maps. We begin by recalling the definition of a smooth triangulation. Let $N$ be a smooth manifold with a piecewise linear structure. A piecewise linear homeomorphism $\kappa \colon P\to N$ from a Euclidean polyhedron $P$ into $N$ is a \emph{piecewise smooth} if there exists a locally finite simplicial complex $\Sigma$ having the properties that $P=|\Sigma|$ and, for each $\sigma \in \Sigma$, there exist a neighbourhood $U\subset \R^{\dim \sigma}$ of $[0,e_1,\ldots,e_{\dim \sigma}]$, a smooth map $\lambda \colon U\to N$ of rank $\dim \sigma$, and an affine map $A\colon [0,e_1,\ldots,e_{\dim \sigma}]\to \sigma$ satisfying $\lambda|_{[0,e_1,\ldots,e_{\dim \sigma}]}=\kappa \circ A$. We say that the piecewise linear structure of $N$ is compatible with the smooth structure of $N$ if there exists a piecewise linear homeomorphism, which is piecewise smooth, from a Euclidean polyhedron into $N$.

We are now ready to state the main result of this appendix.

\begin{proposition}
\label{prop:pl-is-qr}
Let $M$ and $N$ be closed and oriented piecewise linear Riemannian $n$-manifolds, $n\ge 2$. Let $p\colon M\to N$ be a non-degenerate orientation preserving piecewise linear map. If the piecewise linear structure is compatible with the smooth structure on both $M$ and $N$, then $p$ is a quasiregular map.
\end{proposition}

The following result yields the Sobolev-regularity in Proposition \ref{prop:pl-is-qr}. While the result is well-known to experts, we present its proof for the reader's convenience.

\begin{lemma}
\label{lem:sobolev}
Let $M$ be a closed piecewise linear smooth $n$-manifold and let $\kappa \colon P\to M$ be a piecewise linear homeomorphism, which is piecewise smooth. Let $\Sigma$ be a locally finite simplicial complex satisfying $P=|\Sigma|$ and let $\rho \colon M\to \R$ be a continuous function having the property that, for each $n$-simplex $\sigma \in \Sigma$, the restriction $\rho|_{\kappa(\sigma)}$ has a smooth extension. Then $\rho \in W^{1,\infty}(M,\R)$.
\end{lemma}

\begin{proof}
Since $M$ is closed and $\kappa$ is a homeomorphism, $\Sigma$ has finitely many $n$-simplices $\sigma_1,\ldots,\sigma_j$. Without loss of generality, we may assume that each restriction $\kappa|_{\sigma_i}$, $i=1,\ldots,j$, is a smooth map of rank $n$. For each $i\in \{1,\ldots,j\}$, there exist a neighbourhood $U_i \subset M$ of $\kappa(\sigma_i)$ and a smooth function $\psi_i \colon U_i \to \R$ satisfying $\psi_i|_{\kappa(\sigma_i)}=\rho|_{\kappa(\sigma_i)}$.

Define a differential $1$-form $\alpha \colon M\to TM^*$ by setting $\alpha_x=(d\psi_i)_x$ if $x\in \kappa(\interior \, \sigma_i)$ and $\alpha_x=0$ otherwise. Clearly, $\rho \in L^\infty(M)$ and $\alpha \in L^\infty(M,TM^*)$. Thus, it suffices to prove that $\alpha$ is the weak differential of $\rho$.

Let $\varphi \in \Omega_0^{n-1}(M)$. Then
\begin{align*}
\int_M \rho \wedge d\varphi + \int_M \alpha \wedge \varphi &= \sum_{i=1}^j \left( \int_{\kappa(\sigma_i)} \psi_i \wedge d\varphi + \int_{\kappa(\sigma_i)} d\psi_i \wedge \varphi \right) \\
&= \sum_{i=1}^j \int_{\partial \kappa(\sigma_i)} \psi_i \wedge \varphi = \sum_{i=1}^j \int_{\kappa(\partial \sigma_i)} \rho \wedge \varphi
\end{align*}
by Stokes' theorem on manifolds with corners; see e.g. \cite[Theorem 16.25]{LE}.

For each $i$, let $\delta_1^i,\ldots,\delta_{n+1}^i$ be the $(n-1)$-faces of $\sigma_i$. Then
\[
\sum_{i=1}^j \int_{\kappa(\partial \sigma_i)} \rho \wedge \varphi = \sum_{i=1}^j \sum_{j=1}^{n+1} \int_{\kappa(\delta_j^i)} \rho \wedge \varphi.
\]
For each face $\delta_\ell^i$, there exists a unique face $\delta_{\ell'}^{i'}$ such that the faces $\delta_\ell^i$ and $\delta_{\ell'}^{i'}$ agree as sets but have opposite orientations. Thus, the terms $\int_{\kappa(\delta_\ell^i)} \rho \wedge \varphi$ and $\int_{\kappa(\delta_{\ell'}^{i'})} \rho \wedge \varphi$ cancel each other in the sum. Hence, the whole sum vanishes and the claim follows.
\end{proof}

\begin{proof}[Proof of Proposition \ref{prop:pl-is-qr}]
Let $\kappa_1 \colon P_1 \to M$ and $\kappa_2 \colon P_2 \to N$ be piecewise linear homeomorphisms, which are piecewise smooth. Let $\Sigma_1$ and $\Sigma_2$ be finite simplicial complexes for which $P_i=|\Sigma_i|$, $\kappa_i$ is smooth and full rank with respect to $\Sigma_i$, and the map $\kappa_2^{-1} \circ p\circ \kappa_1 \colon P_1 \to P_2$ is simplicial with respect to $\Sigma_1$ and $\Sigma_2$. Let $\sigma_1,\ldots,\sigma_j$ be the $n$-simplices of $\Sigma_1$.

Since $p$ is non-degenerate, each image $\kappa_2^{-1}(p(\kappa_1(\sigma_i)))$ is an $n$-simplex in $\Sigma_2$. Then, for $i=1,\ldots,j$, there exist a neighbourhood $U_i \subset M$ of $\kappa_1(\sigma_i)$, a neighbourhood $U_i' \subset N$ of $p(\kappa_1(\sigma_i))$, and a diffeomorphism $\varphi_i \colon U_i \to U_i'$ satisfying $\varphi_i|_{\kappa_1(\sigma_i)}=p|_{\kappa_1(\sigma_i)}$ since the maps $\kappa_1|_{\sigma_i}$ and $\kappa_2|_{\kappa_2^{-1}(p(\kappa_1(\sigma_i)))}$ are injective smooth maps of rank $n$ and $\kappa_2^{-1} \circ p\circ \kappa_1$ maps $\sigma_i$ smoothly to $\kappa_2^{-1}(p(\kappa_1(\sigma_i)))$.

Let $\iota \colon N\to \R^m$ be a Nash embedding. Then $p\circ \iota \in W^{1,\infty}(M,\R^m)$ by Lemma \ref{lem:sobolev}. Thus, $p\in W^{1,\infty}(M,N)$ and it remains to show that $p$ has finite distortion.

For almost every $x\in M$, we have $x\in \interior \, \kappa(\sigma_i)$ for some $i$ and hence
\[
\frac{\norm{Dp(x)}^n}{\det Dp(x)} \le \frac{\max_{\kappa(\sigma_i)} \norm{D\varphi_i}^n}{\min_{\kappa(\sigma_i)} \det D\varphi}.
\]
Since
\[
\max_{1\le i\le j} \frac{\max_{\kappa(\sigma_i)} \norm{D\varphi_i}^n}{\min_{\kappa(\sigma_i)} \det D\varphi} <\infty,
\]
we conclude that $p$ has finite distortion. This concludes the proof.
\end{proof}


\bibliographystyle{abbrv}

\end{document}